\pdfoutput=1
\documentclass[12pt]{article}

\usepackage[reqno]{amsmath} 
\usepackage{amssymb}
\usepackage{amstext}
\usepackage{amsthm}
\usepackage{amsfonts}
\usepackage{mathtools}
\usepackage{enumerate}  
\usepackage{dsfont}
\usepackage[utf8]{inputenc}

\usepackage[backend=biber,style=numeric,doi=false,isbn=false,url=false,hyperref=true,backref=true,maxcitenames=9,maxnames=9]{biblatex}
\bibliography{pstpolytime.bib}

\usepackage{hyperref}
\usepackage[pdftex]{graphicx} 

\theoremstyle{plain}

\swapnumbers

\newtheorem{theorem}{Theorem}[section]
\newtheorem{lemma}[theorem]{Lemma}

\theoremstyle{definition}

\usepackage{notation}

\title{Perfect state transfer is poly-time}
\author{  Gabriel Coutinho\footnote{Acknowledges the support of grants FAPESP 15/16339-2 and FAPESP 13/03447-6. Affiliation at the time of original work: Dep. Ciência da Computação / IME-USP / São Paulo, Brazil.} \\ \small Dep. Ciência da Computação \\ \small UFMG \\ \small Belo Horizonte, MG, Brazil \\ \small \texttt{gabriel@dcc.ufmg.br} \and Chris Godsil\footnote{Acknowledges the support of NSERC Grant RGPIN-9439.} \\ \small Dep. Combinatorics and Optimization \\ \small University of Waterloo \\ \small Waterloo, ON, Canada \\ \small \texttt{ cgodsil@uwaterloo.ca }}  

\date{  \today  }

\begin{document}

\maketitle

\begin{abstract}
We show that deciding whether a graph admits perfect state transfer can be done in polynomial time on a classical computer with respect to the size of the graph.
\end{abstract}

\hspace{-0.325in}\begin{tabular}{rl}
\textbf{Keywords:} & Perfect state transfer; quantum walks; spectral graph theory. \\

\textbf{MSC:} & \texttt{81P68; 05C50; 15A16.}
\end{tabular}

\section{Introduction}

Since the seminal papers by Bose \cite{BoseQuantumComPaths} and Christandl et al.~\cite{ChristandlPSTQuantumSpinNet}, it is safe to say that over 500 papers have been published in physics, computer science and mathematics journals concerning the topic of state transfer over continuous-time quantum walk models. From the combinatorics point of view, there has been a remarkable effort in characterizing simple graphs for which perfect state transfer occurs according to the adjacency matrix or Laplacian matrix models. State transfer has been studied in distance-regular graphs (Coutinho et al.~\cite{CoutinhoGodsilPSTProducts}, Jafarizadeh and Sufiani \cite{JafarizadegPSTDRG}), graph products and alike structures (Ge et al.~\cite{TamonGeGreenbergPerezTamonPSTproducts}, Coutinho and Godsil \cite{CoutinhoGodsilPSTProducts}, Pal and Bhattacharjya \cite{PalBhattacharjyaPSTNEPSP3}, Ackelsberg et al.~\cite{TamonPSTCoronas}), weighted paths (Kay \cite{KayReviewPST}, Vinet and Zhedanov \cite{VinetZhedanovHowTo}), Cayley graphs for abelian groups (Saxena el al.~\cite{SaxenaSeveriniShparlinski}, Cheung and Godsil \cite{GodsilCheungPSTCubelike}, Chan \cite{AdaChanComplexHadamardIUMPST}, Ba\v{s}i\'{c}\cite{BasicCirculant}), and trees (Coutinho and Liu \cite{CoutinhoLiu}), among other examples.

However, to the best of our knowledge, the classical complexity and related computational aspects of determining when perfect state transfer occurs in simple graphs has not been studied. This is the topic of this paper.

\section{The hardness of perfect state transfer}

Let $M$ be a symmetric integer matrix whose rows and columns are indexed by a finite set $V$ of $n$ elements, and whose spectral decomposition is given by
\[M = \sum_{r = 0}^d \theta_r E_r,\]
where the $\theta_r$'s are the distinct eigenvalues of $M$.

If $a,b \in V$, we say that $M$ admits perfect state transfer between $a$ and $b$ if there is a time $\tau \in \R^+$ such that 
\begin{align}|\exp( \ii \tau M)_{a,b}| = 1. \label{pst} \end{align}

Denote by $\ee_a$ the vector whose entires are indexed by $V$ with its $a$th entry equal to $1$, and all the remaining equal to $0$. A necessary condition (see Godsil \cite[Lemma 5.1]{GodsilPerfectStateTransfer12}) for perfect state transfer to occur is that, for all $r$, 
\[E_r \ee_a = \pm E_r \ee_b.\]
When this occurs, we say that $a$ and $b$ are strongly cospectral. The eigenvalue support of $a$, denoted by $\Theta_a$, is the set of eigenvalues of $M$ whose projection of $\ee_a$ onto the corresponding eigenspace is not $0$, that is
\[\Theta_a = \{\theta_r : E_r \ee_a \neq 0\}.\]

If $M$ is a matrix that encodes the adjacency of a graph, say the adjacency or the Laplacian matrix, we say that the graph admits perfect state transfer between the corresponding vertices, or that such vertices are strongly cospectral, with respect to the quantum walk model determined by $M$.

In this paper, we show the following result, which can be notably applied for the cases where $M$ is the adjacency matrix or the Laplacian matrix of a graph.
\begin{theorem} \label{1.1}
For any $n \times n$ symmetric integer matrix $M$ whose entries belong to\footnote{Or, more generally, entries whose absolute values are bounded by a polynomial in $n$.} $[-n,n]$, deciding whether or not $M$ admits perfect state transfer can be done in polynomial time in $n$.
\end{theorem}
To show this result, we make use of the following characterization of perfect state transfer.
\begin{theorem} (see Coutinho \cite[Theorem 2.4.4]{CoutinhoPhD})\label{thm:charpst}
If $M$ is a symmetric integer matrix and two of its columns are indexed by $a$ and $b$, and if $\theta_0>...>\theta_k$ are the eigenvalues in $\Theta_a$, then $M$ admits perfect state transfer between $a$ and $b$ if and only if the following conditions hold.
\begin{enumerate}[(a)]
\item Columns $a$ and $b$ are strongly cospectral. \label{a}
\item Columns $a$ and $b$ are periodic, or equivalently (see Godsil \cite[Theorem 6.1]{GodsilPerfectStateTransfer12}), non-zero elements of $\Theta_a$ are either all integers or all quadratic integers. Moreover, there is a square-free integer $\Delta$, an integer $p$, and integers $q_0,...,q_k$ such that
\[\theta_r = \frac{p + q_r \sqrt{\Delta}}{2}\ , \quad \text{for all }r = 0,...,k.\] \label{b}
\item \label{c} Let $g = \gcd  \{ (\theta_0 - \theta_r)/\sqrt{\Delta}  \}_{r=0}^k$. Then
\begin{enumerate}
\item $E_r \ee_a = E_r \ee_b$ if and only if $(\theta_0-\theta_r)/(g \sqrt{\Delta})$ is even, and
\item $E_r \ee_a = - E_r \ee_b$ if and only if $(\theta_0-\theta_r)/(g \sqrt{\Delta})$ is odd.
\end{enumerate}
\end{enumerate}
Moreover, if these conditions hold, the minimum time perfect state transfer occurs is $\pi / (g \sqrt{\Delta})$. \qed
\end{theorem}

In order to decide whether perfect state transfer occurs, we must check all three conditions above. On the course of the following results, we will show how each of the conditions can be independently checked in polynomial time.

Let $\phi(x)$ be the characteristic polynomial of $M$, and if $T \subset V$, let $\phi_T(x)$ be the characteristic polynomial of the matrix obtained from $M$ by removing the rows and columns indexed by $T$. Recall that the complexity of computing the characteristic polynomial is the same as computing the determinant.

\begin{lemma}\label{lem:multpoles}
For $j \in \{0,...,d\}$, the multiplicity of $\theta_j$ as a pole of $\phi_T(x) / \phi(x)$ is at most equal to the rank of $(E_j)_{T,T}$. If the rank is equal to $|T|$, then the multiplicity of the pole is also equal to $|T|$.
\end{lemma}
\begin{proof}
The following identity is due to Jacobi (see for instance Godsil \cite[Section 4.1]{GodsilAlgebraicCombinatorics}):
\[\det [(xI - M)^{-1}]_{T,T} = \frac{\phi_T(x)}{\phi(x)}.\]
On the other hand, using the spectral decomposition of $M$, it follows that
\[[(xI - M)^{-1}]_{T,T} = \sum_{r = 0}^d \frac{1}{x - \theta_r} (E_r)_{T,T}.\]
Let us just denote $F_r = \frac{1}{x - \theta_r} (E_r)_{T,T}$ and $H = \sum_{r \neq j} F_r$. Now let $P$ be an orthogonal matrix that diagonalizes $F_j$ to a diagonal matrix $D$. It follows that
\[\frac{\phi_T(x)}{\phi(x)} = \det \left(  D + P^T H P \right).\]
This determinant is the sum of the determinants of the matrices we obtain from $P^T H P$ by replacing each subset of columns by the corresponding set from $D$. Note that $D$ contains poles at $\theta_j$ but $P^THP$ does not. Thus the multiplicity of $\theta_j$ as a pole of $\phi_T(x) / \phi(x)$ cannot exceed the rank of $(E_j)_{T,T}$. While equality does not hold in general, it holds for the full rank case, as a consequence of \cite[Theorem 7.2.9]{cvetkovic1997eigenspaces}.
\end{proof}

\begin{lemma} \label{lem:condition(a)}
Given $M$ an $n \times n$ symmetric matrix whose columns are indexed by $V$, two columns $a$ and $b$ of $M$ are strongly cospectral if and only if 
\begin{enumerate}[(i)]
\item $\phi_a(x) = \phi_b(x)$ (when this occurs we call $u$ and $v$ cospectral), and
\item The poles of $\phi_{ab}(x) / \phi(x)$ are simple.
\end{enumerate}
As a consequence, condition (\ref{a}) of Theorem \ref{1.1} can be checked in polynomial time on $n$.
\end{lemma}
\begin{proof} \
\begin{itemize}
\item[] Claim 1: Condition (ii) is equivalent to $a$ and $b$ being parallel, that is, for all $r$, either $E_r \ee_a$ or $E_r \ee_b$ is a scalar multiple of the other.

This is an immediate consequence of Lemma \ref{lem:multpoles}.

\item[] Claim 2: Condition (i) is equivalent to, for all $r$, $(E_r)_{a,a} = (E_r)_{b,b}$.

This is due to Godsil and McKay \cite{GodsilMcKay}.
\end{itemize}

Now, clearly if $a$ and $b$ are strongly cospectral, then they are parallel. So we proceed to assume that $a$ and $b$ are parallel vertices and show that, in this case, strong cospectrality is equivalent to cospectrality. In fact, $E_r$ is idempotent, thus ${E_r \ee_a = 0}$ if and only if $(E_r)_{a,a} = 0$. So in either case it must be that $E_r \ee_a = 0$ if and only if $E_r \ee_b = 0$ for all $r$. Now suppose that $E_r \ee_a = \lambda E_r \ee_b \neq 0$. Then 
\[(E_r)_{a,a} = \lambda (E_r)_{a,b} = \lambda (E_r)_{b,a} = \lambda^2 (E_r)_{b,b},\]
thus they are strongly cospectral if and only if they are cospectral.

To finish the proof, we show that both conditions can be checked in polynomial time. Nothing needs to be said about condition (i).  For condition (ii), let $g(x)$ be the greatest common divisor of $\phi(x)$ and $\phi_{ab}(x)$. Condition (ii) is equivalent to $f(x) = \phi(x) / g(x)$ being a polynomial without repeated roots, which in turn can be verified by checking whether or not the greatest common divisor of $f(x)$ and $f'(x)$ is a constant. 
\end{proof}

\begin{lemma} \label{lem:condition(b)}
Given an $n \times n$ symmetric integer matrix $M$ whose entries belong to $[-n,n]$, condition (b) of Theorem \ref{thm:charpst} can be tested in polynomial time. Moreover, if the conditions passes, the eigenvalues in $\Theta_a$ and their corresponding eigenvectors can be computed in polynomial time.
\end{lemma}
\begin{proof}
For a column $a$ of $M$, let
\[f(x) = \frac{\phi(x)}{\gcd(\phi(x) , \phi_a(x))}.\]
From Lemma \ref{lem:multpoles}, we see that $f(x)$ has simple roots, which are precisely the eigenvalues of $M$ that lie in $\Theta_a$. If the degree of $f(x)$ is $k$, the coefficient of $(-x)^{k-1}$ is the sum of the roots of $f(x)$. 

From the assumption that the entries of $M$ are bounded, we have that
\[\sum_{r = 0}^d \theta_r^2 \leq \tr M^2 = \sum_{ 1 \leq i,j \leq n} M_{ij}^2 \leq n^4,\]
thus the eigenvalues of $M$ lie in the interval $[-n^4,n^4]$. Condition (b) holds if and only if all roots of $f(x)$ are either integers or quadratic integers with a constant rational term. Because the sum of the eigenvalues is known, the number of possible candidates for the roots of $f(x)$ in the interval $[-n^4,n^4]$ that satisfies condition (b) is a polynomial function on $n$, therefore we can efficiently check whether $f(x)$ satisfies condition (b), and compute its factorization.
\end{proof}

\begin{proof}[Proof of Theorem \ref{1.1}]
From Lemmas \ref{lem:condition(a)} and \ref{lem:condition(b)}, conditions (a) and (b) of Theorem \ref{thm:charpst} can be checked efficiently, and if condition (b) holds then $\Theta_a$ can be computed efficiently. Upon knowing the elements of $\Theta_a$, the computation of $g$ and of the ratios $(\theta_0 -\theta_r) / (g \sqrt{\Delta})$ can be carried out, remaining only to check whether the signs appearing in condition (c) are compatible. To that effect, it suffices to note that for any eigenvector $\xx$ associated to an eigenvalue $\theta_r$ such that $\xx_a \neq 0$, we have that if $E_r \ee_a = E_r \ee_b$, then $\xx^T E_r \ee_a = \xx^T E_r \ee_b$ and thus $\xx_a = \xx_b$; and likewise, if $E_r \ee_a = - E_r \ee_b$, then $\xx_a = -\xx_b$.
\end{proof}

\section{Comments and related open questions}

\subsubsection*{Related work}

Chen et al.~\cite{ChenYuTanContinuousOrbitPolyTime} recently worked on the problem of determining when there is a positive real number $t$ such that, for a given rational matrix $M$ and rational vectors $\uu$ and $\vv$, $\e^{t M} \uu = \vv$. There is a clear resemblance to our problem, however our result does not follow from theirs for the following reason. They assume one can compute the Jordan Normal Form and the eigenvectors of a given matrix efficiently, citing the work of Cai \cite{CaiJNFPolyTime}, who in turn deals with roots of polynomials of degree larger than $4$ by computing sufficiently good rational approximations. We note that this approach would not determine whether or not perfect state transfer occurs, but only some arbitrarily good phase-shifted version of it. The fact that we are able to examine the eigenvalues in the support of a column locally (Lemma \ref{lem:condition(b)}), and restrict to the case where such eigenvalues are integers or quadratic integers (Godsil \cite[Theorem 6.1]{GodsilPerfectStateTransfer12}) is key to efficiently determine if perfect state transfer indeed occurs. Also, the fact that they are not considering complex numbers leads them to easily rule out the case $p \neq 0$ of condition (b) in Theorem \ref{thm:charpst}, whereas we know that there are infinitely many examples of symmetric matrices admitting perfect state transfer in this case.

\subsubsection*{Uniform mixing}
Given a $n \times n$ integer symmetric matrix $M$, consider the mixing matrix\footnote{This is usually denoted as $M(t)$, but we are already using $M$ with a different meaning.}
\[N(t) = \exp(\ii t M) \circ \exp(-\ii t M),\]
where the product depicted is the entry-wise matrix product\footnote{Also known as Schur or Hadamard product, or perhaps the bad student product.}. We say that $M$ admits uniform mixing at a time $\tau$ if
\[N(\tau) = \frac{1}{n} J,\]
where $J$ stands for the matrix whose all entries are equal to $1$.

Determining whether a matrix admits uniform mixing is a considerably harder problem than that of perfect state transfer, even in the special case where $M$ is the adjacency matrix of a simple graph and its eigenvalues are known (see for instance Godsil et al.~\cite{GodsilMullinRoy}). For this reason, we believe a reasonable question here is to determine classes of graphs for which the existence of uniform mixing can be tested in polynomial time. For example:
\begin{itemize}
\item[1.] Can we test whether a cubelike graph admits uniform mixing in polynomial time?
\end{itemize}

\subsubsection*{Average mixing}

Note that $N(t)$ denotes a probability distribution. The average of this distribution can be computed as follows
\[\hat{N}(t) = \lim_{T \to \infty} \frac{1}{T} \int_{0}^T N(t) \dd t.\]
As shown in Godsil \cite{GodsilAverageMixing}, it follows that, if $M$ admits spectral decomposition $M = \sum_{r = 0}^d \theta_r E_r$, then
\begin{align}\hat{N}(t) = \sum_{r = 0}^d E_r \circ E_r. \label{eq2}\end{align}
The average mixing matrix enjoys interesting properties, notably, it is rational, doubly-stochastic and positive semidefinite. Moreover, two columns of $\hat{N}(t)$ are equal if and only if they are strongly cospectral, a property that, as seen above, can be tested in polynomial time. We therefore ask:
\begin{itemize}
\item[2.] Given $M$ a symmetric integer matrix, can $\hat{N}(t)$ be computed in polynomial time? 
\end{itemize}
Naturally, as seen from (\ref{eq2}), $\hat{N}(t)$ can be computed numerically in polynomial time up to desired precision. However the rationality of $\hat{N}$ suggests there could be a way of doing this computation symbolically.

\subsubsection*{Pretty good state transfer}

Recall the definition of perfect state transfer in $(\ref{pst})$. This definition can be relaxed to an $\epsilon$ version as follows. If $a,b \in V$, we say that $M$ admits pretty good (or almost) state transfer between $a$ and $b$ if for any $\epsilon > 0$, there is a time $\tau \in \R^+$ such that 
\[|\exp( \ii \tau M)_{a,b}| \geq 1 - \epsilon. \]
Naturally, perfect state transfer implies pretty good, but the converse does not hold. For instance, it is been shown that infinitely many paths admit pretty good state transfer between the end vertices in both the adjacency matrix and the Laplacian matrix model of quantum walks (see \cite{GodsilKirklandSeveriniSmithPGST} \cite{BanchiCoutinhoGodsilSeverini}), as opposed to what happens with perfect state transfer. It is thus natural to ask whether pretty good state transfer can also be determined in polynomial time. As it turns out, pretty good state transfer has been characterized via Kronecker's theorem on Diophantine approximations (see \cite[Theorem 4]{BanchiCoutinhoGodsilSeverini}). In order to test the condition on Kronecker's theorem, it seems to be necessary to obtain some information about the minimal polynomial over the rationals of each of the eigenvalues of $M$, and their corresponding splitting fields. This suggests that the problem of deciding whether or not pretty good state transfer occurs might be harder than deciding perfect state transfer. We thus ask:
\begin{itemize}
\item[3.] Given $M$ a symmetric integer matrix, is there an algorithm that terminates in finite time that determines whether or not there is pretty good state transfer between two columns of $M$?
\end{itemize}

\printbibliography

\end{document}